\documentclass[12pt]{amsart}
\usepackage{amsthm}
\usepackage{amsmath}
\usepackage[margin=1.25in]{geometry}
\usepackage{amssymb}
\usepackage{verbatim}
\usepackage{longtable}
\usepackage{tikz-cd}
\usepackage{color}
\usepackage{stmaryrd}
\usepackage{graphicx}
\usepackage{latexsym}   
\usepackage{amsmath}    
\usepackage{amsbsy}
\usepackage{array}
\usepackage[all]{xy}
\usepackage{caption}
\usepackage{enumerate}
\usepackage{mathabx}

\theoremstyle{plain}
\newtheorem{algorithm}{Algorithm}[section]

\newtheorem{corollary}[algorithm]{Corollary}

\newtheorem{definition}[algorithm]{Definition}

\newtheorem{lemma}[algorithm]{Lemma}
\newtheorem{sublemma}[algorithm]{Sublemma}
\newtheorem{question}[algorithm]{Question}

\newtheorem*{theorem*}{Theorem}
\newtheorem{theorem} [algorithm] {Theorem}

\newtheorem{proposition}[algorithm]{Proposition}
\newtheorem{remark}[algorithm]{Remark}

\numberwithin{equation}{algorithm}

\newtheorem*{BC}{Bott Conjecture}
\newtheorem*{MSRC}{Maximal Symmetry Rank Conjecture for Almost Non-negative Curvature}

\newtheorem*{thmA}{Theorem A}

\newtheorem*{corB}{Corollary B}

\newtheorem*{ST}{Slice Theorem}

\usepackage{epsfig}
\usepackage[]{graphicx}
\graphicspath{ {./Images/} }
\usepackage{epstopdf}

\usepackage{hyperref}
\usepackage{listings}
\usepackage{enumerate}
\usepackage{caption}
\usepackage{subcaption}
\usepackage{float}

\newtheoremstyle{redefinition}
    {}{}
    {}{}
    {\bfseries}{.}
    { }
    {\thmname{#1}\thmnumber{ #2}\thmnote{ [#3]}}
\theoremstyle{redefinition}
\newcommand{\R}{\mathbb{R}}

\newcommand{\C}{\mathbb{C}}

\DeclareMathOperator{\diam}{diam}
  
\DeclareMathOperator{\curv}{curv}

\usepackage[T1]{fontenc}
\usepackage{makecell}

\counterwithout{equation}{section}

\title[Almost non-negatively curved 4- 5- and 6-manifolds]{Maximal symmetry rank and almost non-negative curvature in low dimensions}
\author{Samuel Bartel}\address{\hspace{-.1in}Department of Mathematics, Oregon State University, Corvallis, OR 97331 USA}\email{bartesam@oregonstate.edu}
\date{\today}

\begin{document}
\begin{abstract}   We establish an upper bound for the symmetry rank of closed, simply connected, almost non-negatively curved manifolds of dimension $4$ through $9$ and in the case of maximal symmetry rank provide an equivariant diffeomorphism classification of this class of manifolds in dimensions 4, 5, and 6.
\end{abstract}
\maketitle

\section{Introduction}

A smooth manifold is said to be almost non-negatively curved if it admits a sequence of Riemannian metrics with lower sectional curvature bounds approaching zero and a uniform diameter upper bound, as described in Definition \ref{ANNCdef}. Via Proposition \ref{ANNC}, we see that this property is equivalent to the manifold collapsing to a point while maintaining a uniform lower sectional curvature bound. A classification of such manifolds would therefore be a key step toward understanding collapse on a broader scale. However, a classification of such manifolds is only known in dimensions 2 and 3. The former follows from the Gauss-Bonnet Theorem, and the latter is due to Shioya and Yamaguchi \cite{ShioyaYamaguchi}.

The Grove Symmetry Program seeks to classify manifolds with a lower curvature bound by additionally assuming ``large'' symmetries. Though originally introduced for positive and non-negative curvature, this course of inquiry has also proven useful for other lower curvature bounds, including almost non-negative curvature. For example, Harvey and Searle \cite{Harvey_Searle} obtain a classification in dimension four by assuming the existence of an isometric circle or $T^2$-action. We apply this principle to almost non-negatively curved manifolds of dimensions four, five, and six, yielding a new proof for dimension four.

\begin{thmA}\label{thmA}
    Let $T^k$ act smoothly and effectively on a closed, smooth, simply connected $n$-manifold $M$, $4\leq n\leq6$. Then the following hold.
	\begin{enumerate}[(1)]
		\item{$k\leq\lfloor\frac{2n}{3}\rfloor=n-2$.}
		\item{If $k=n-2$ and $M$ is an almost non-negatively curved $T^{n-2}$-manifold, then $M$ is equivariantly diffeomorphic to $\mathcal{Z}/T^m$, where 
        $$\mathcal{Z}=\begin{cases}
            S^3\times S^3 \,\,\mathrm{and}\,\, 0\leq m\leq 2,\\
            S^5  \,\, \mathrm{and} \,\, 0\leq m\leq 1, \,\,\\
            S^4\,\,\mathrm{and}\,\,m=0,\\
            \end{cases}$$
        and the $T^m$-action on any of these spaces is free and linear.} 
	\end{enumerate}
\end{thmA}
\noindent In the statement of \hyperref[thmA]{Theorem A}, $T^0$ refers to the trivial group $\{e\}$.

\begin{remark}\label{rem1}
    The manifolds that occur in Part (2) of Theorem A are equivariantly diffeomorphic to one of $S^4,S^2\times S^2, \C P^2,\C P^2\#\pm\C P^2,S^5,S^3\times S^2,S^3\tilde{\times}S^2,$ or $S^3\times S^3$, with a corresponding linear $T^{n-2}$-action. See, for example, Galaz-Garc\'ia and Searle \cite{Galaz_Garcia_Searle_2011} and Galaz-Garc\'ia and Kerin \cite{Galaz_Garcia_Kerin}.
\end{remark}

We also obtain the following corollary of \hyperref[thmA]{Theorem A}.

\begin{corB}\label{corB}
    Let $M^n$ be a closed, simply connected, almost non-negatively curved $T^k$-manifold, $7\leq n\leq9$. Then $k\leq \lfloor\frac{2n}{3}\rfloor =n-3$.
\end{corB}

\begin{remark}
    There are known examples of almost non-negatively curved $n$-manifolds, $7\leq n\leq 9$, which admit symmetry rank $n-3$. For example, $S^3\times S^4$, $S^3\times S^5$, and $S^3\times S^3\times S^3$ each admit a non-negatively curved metric invariant under a $T^{4}$-, $T^5$-, and $T^6$-action, respectively. Therefore $n-3$ is the maximal symmetry rank for these dimensions.
\end{remark}

A longstanding conjecture of Bott relates non-negative curvature to the topology of a closed, simply connected Riemannian manifold. In particular, the conjecture states that such a manifold $M$ is rationally elliptic, that is, $\dim \pi_*(M)\otimes\mathbb{Q}<\infty$ as a rational vector space. \hyperref[thmA]{Theorem A} provides evidence to support the following extension of the conjecture to manifolds of almost non-negative curvature due to Grove and Halperin \cite{Grove_2002}, as all of the manifolds listed in Remark \ref{rem1} are known to be rationally elliptic.

\begin{BC}\cite{Grove_2002}
    Let $M$ be a closed, simply connected, almost non-negatively curved manifold. Then $M$ is rationally elliptic.
\end{BC}

\begin{remark}
    The 4- and 5-manifolds obtained in \hyperref[thmA]{Theorem A} are the only rationally elliptic manifolds in those dimensions up to diffeomorphism by Paternain and Petean \cite{Paternain_Petean}. However, the class of rationally elliptic 6-manifolds contains manifolds other than $S^3\times S^3$, see Hermann \cite{Hermann}.
\end{remark}

Theorem A and Corollary B also provide evidence for extending the Maximal Symmetry Rank Conjecture for non-negative curvature as stated in Escher and Searle \cite{Escher_Searle} to almost non-negative curvature. Namely, we make the following conjecture.

\begin{MSRC}
Let $M^n$ be a closed, simply connected, almost non-negatively curved $T^k$-manifold. Then  the following hold:
\begin{enumerate}
\item 
$k\leq \lfloor \frac{2n}{3}\rfloor$; and 

\item When $k= \lfloor \frac{2n}{3}\rfloor$, $M^n$ is equivariantly diffeomorphic to  
$\mathcal{Z}/T^m$ with a linear $T^k$-action,  where 
$$\mathcal{Z}=  \prod_{i\leq r} S^{2n_i-1} \times\prod_{i>r} S^{2n_i},$$
 with  $n_i\geq 2,  \,\,0 \leq m \leq 2n \mod 3,$
and the $T^m$-action on $\mathcal{Z}$ is  free and linear.
 \end{enumerate}
\end{MSRC}

\subsection{Organization} The organization of the paper is as follows. In Section 2, we introduce definitions and notation. In Section 3, we prove \hyperref[thmA]{Theorem A} and \hyperref[corB]{Corollary B}.

\subsection{Acknowledgements} The author was partially supported by Catherine Searle's NSF Grant DMS-2204324. The author thanks his current PhD advisors Christine Escher and Catherine Searle for their helpful comments. This article is based on work from the author's master's thesis \cite{Bartel}, completed at Wichita State University under the supervision of Catherine Searle. 

\section{Preliminaries}

In Section 2.1 we introduce definitions and notation for compact transformation groups on closed manifolds. In Section 2.2 we introduce topological results for smooth torus actions. Section 2.3 provides definitions and notation for Alexandrov spaces. Finally, Section 2.4 introduces results that depend on the geometry of almost non-negative curvature.

\subsection{Transformation groups}

Let $G$ be a compact Lie group acting on a closed smooth manifold $M$. The \textit{isotropy group} at $x\in M$ is the group $G_x=\{g\in G\mid g\cdot x=x\}$, and the \textit{orbit} of $x$ is the set $G(x)=\{y\in M\mid y=g\cdot x\text{ for some }g\in G\}$. The $G$-action is \textit{effective} if $\bigcap_{x\in M}G_x=\{e\}$. An orbit is \textit{principal} if its isotropy group is minimal with respect to containment of conjugates. Orbits are called \textit{exceptional} if they are not principal but have isotropy group of the same dimension as a principal orbit. Orbits which are not principal or exceptional are called \textit{singular}.

Given a smooth action of a compact Lie group, $G$, on a smooth manifold $M$, the \textit{orbit space} is $\overline{M}:=M/G$, where $x\sim y$ if $y\in G(x)$. If $M$ is equipped with a Riemannian metric, $\overline{M}$ inherits the orbital distance metric $\overline{d}(\overline{x},\overline{y}):=\min\{d_M(p_1,p_2)\mid p_1\in G(x),p_2\in G(y)\}$. The \textit{orbital projection map} $\pi:M\to \overline{M}$ is defined by $\pi(x)=\overline{x}$. Similarly, we denote $\pi(A)=\overline{A}$ for an arbitrary set $A\subset M$.

We say an action is of \textit{cohomogeneity $m$} if the orbit space is $m$-dimensional. By work of Mostert \cite{Mostert}, the class of cohomogeneity one manifolds is in one-to-one correspondence with cohomogeneity one group diagrams, which are ordered tuples recording isotropy information. When $\overline{M}$ is homeomorphic to an interval, the \textit{cohomogeneity one group diagram} is the ordered quadruple $(G,H,K_1,K_2)$, where $H$ is the principal isotropy group corresponding to interior points of $\overline{M}$ and $K_1$ and $K_2$ are the isotropy groups corresponding to the endpoints of $\overline{M}$.

The \textit{fixed point set} of the $G$-action is denoted
$$M^G:=\{x\in M\mid g\cdot x=x\text{ for all }g\in G\},$$
and $\dim(M^G)$ is defined to be the maximal dimension among components of $M^G$. The dimension of $\overline{M}$ is constrained by that of $M^G$, in particular $\dim(\overline{M})\geq\dim(M^G)+1$ for nontrivial $G$-actions. In the case where $\dim(\overline{M})=\dim(M^G)+1$ the action is said to be \textit{fixed-point homogeneous}. Grove and Searle \cite{Grove_Searle_97} classify positively curved, simply connected, fixed-point homogeneous manifolds. Non-negatively curved, simply connected, fixed-point homogeneous manifolds are classified in Galaz-Garc\'ia \cite{Galaz_Garcia_2012} for dimensions three and four and Galaz-Garc\'ia and Spindeler \cite{Galaz_Garcia_Spindeler} for dimension five.

We now recall the Slice Theorem, which provides a description of the local structure near each orbit. There are many formulations of this theorem, see for example Bredon \cite{Bredon} or Grove \cite{Grove_2002}.

\begin{ST}\label{ST}
    Let $G$ be a compact Lie group acting smoothly on $M$. Then for any $x\in M$, a sufficiently small tubular neighborhood of $G(x)$ is equivariantly diffeomorphic to $G\times_{G_x}D_x^\perp$, where $D^\perp_x$ is the unit normal disk to the orbit $G(x)$ at $x$.
\end{ST}
\noindent In particular, when the $G$-action is effective the isotropy group $G_x$ also acts effectively on $S^\perp_x$, the normal sphere to the orbit $G(x)$ at $x$. 

Recall that when $M$ is a Riemannian manifold, we say that $G$ \textit{acts by isometries} if the $G$-action is an isometry for each element of $G$. Theorem 5.1 in Chapter II of Kobayashi \cite{Kobayashi} shows that the fixed point set of an isometric action admits a rich structure of its own.
\begin{theorem}\cite{Kobayashi}\label{fix}
    Let $G$ act on $(M,g)$ by isometries, and let $F$ be the set of points fixed by this action. Then each connected component of $F$ is a closed, totally geodesic submanifold of $M$.
\end{theorem}

Since the action of $G_x$ on the unit normal disk is effective, when $G= T^k$, isotropy groups have the property that $\dim(G_x)\leq n-k$. The case of equality corresponds to a \textit{minimal orbit}, and a $T^k$-action with such an orbit is called \textit{isotropy-maximal}.

\begin{definition}[Isotropy-maximal action]
	Let $M^n$ be a connected manifold with an effective $T^k$-action. The $T^k$-action on $M^n$ is \textnormal{isotropy-maximal} if either of the following equivalent conditions hold:
	\begin{enumerate}[(1)]
		\item{There is a point $x\in M$ such that the dimension of the isotropy subgroup is $n-k$, that is, $\dim(T^k_x)=n-k$; or}
		\item{There is a point $x\in M$ such that $\dim(T^k(x))=2k-n$, in which case the orbit $T^k(x)$ through $x\in M$ is called a \textnormal{minimal orbit}.}
	\end{enumerate}
\end{definition}

\noindent Lemma 2.3 in Ishida \cite{Ishida_2019} describes useful facts about minimal orbits, one of which is the following.

\begin{lemma}\cite{Ishida_2019}\label{minisolated}
    Let $M$ be a connected manifold equipped with an isotropy-maximal action by a compact torus $G$. Then each minimal orbit is isolated.
\end{lemma}

\noindent In addition to having a maximal dimension isotropy subgroup, these actions are maximal with respect to containment for effective torus actions, as seen in Lemma 2.2 of \cite{Ishida_2019}.

\begin{lemma}\cite{Ishida_2019}\label{maxtorus}
	Let $M$ be a connected manifold and let $T^k$ act effectively on $M$. If $T^l$ is a subtorus of $T^k$ such that the action restricted to $T^l$ is isotropy-maximal, then $T^k=T^l$. In particular, $k=l$.
\end{lemma}

\noindent \textit{Locally standard} is another notion of structure for a torus action, see for example Wiemeler \cite{Wiemeler}.

\begin{definition}[Locally standard action]\label{locstd}
	A $T^k$-action on $M^n$ is called \textnormal{locally standard} if for each point $x\in M$, there is a neighborhood of $x$ in $M$ which is $T^k$-equivariantly diffeomorphic to
	$$T^r\times W\times\R^m,$$
	where $r=k-\dim(T^k_x)$, $W$ is a faithful $T^k_x$-representation of real dimension $2\dim(T^k_x)$, and $T^k\cong T^r\times T^k_x$ acts trivially on $\R^m$, $T^r$ acts trivially on $W$, and $T^k_x$ acts trivially on $T^r$.
\end{definition}

\subsection{Topological results}
We first recall several results concerning smooth torus actions of cohomogeneity two, beginning with Theorem 1.3 from Kim, McGavran, and Pak \cite{Kim}.
\begin{theorem}\cite{Kim}\label{ospace}
	Let $T^{n-2}$ act effectively on a closed, simply connected manifold $M^n$, $n\geq4$. Then $\overline{M}$ is homeomorphic to $D^2$, interior points of $\overline{M}$ correspond to orbits with trivial isotropy, and boundary points of $\overline{M}$ correspond to orbits with $T^1$- or $T^2$-isotropy.
\end{theorem}

\noindent We refer to the projection of an orbit with $T^2$-isotropy as a \textit{vertex} of $\overline{M}$ and the segment of boundary between consecutive vertices as an \textit{edge}. Since vertices correspond to minimal orbits, they are isolated by Lemma \ref{minisolated}. We now recall Corollary 1.7 from \cite{Kim}, which provides a lower bound on the number of distinct circle isotropy subgroups.

\begin{corollary}\cite{Kim} \label{sides}
	Let $T^{n-2}$ act effectively on $M^n$, a closed, simply connected $n$-manifold. Then all isotropy subgroups generate the whole $T^{n-2}$, and there are at least $(n-2)$ different circle isotropy subgroups of $T^{n-2}$.
\end{corollary}

\noindent In fact, the number of distinct circle isotropy subgroups provides a lower bound on the number of vertices in the orbit space, since two adjacent arcs in $\partial\overline{M}$ corresponding to orbits fixed by distinct circle subgroups have intersection corresponding to orbits fixed by the $T^2$ generated by these two subgroups. Thus, the intersection consists of vertices. Since there are at least $n-2$ arcs by Corollary \ref{sides}, there must be at least $n-2$ corresponding vertices. Hence, the following lemma holds.

\begin{lemma}\label{vertices}
    Under the conditions of Corollary \ref{sides}, there are at least $n-2$ vertices in $\overline{M}$.
\end{lemma}

Suppose now that $T^1\subset T^{n-2}$ corresponds to the isotropy group of an edge in $\overline{M}$. Then the $T^1$-action on $M$ has $\dim(M^{T^1})=n-2$ and $\dim(M/T^1)=n-1$, so it is fixed-point homogeneous. Thus, all discussion in this article fits into the broader context of manifolds which admit a fixed-point homogeneous circle action.

Angle information at vertices determines when the preimage of an edge is a submanifold, as shown in the proof of Theorem 5.1 in McGowan and Searle \cite{McGowan} and the discussion which immediately follows. We sketch the proof for completeness.

\begin{proposition} \cite{McGowan}
    Let $G$ be a closed Lie group acting smoothly and effectively  by cohomogeneity two on a closed, simply connected manifold $M$  and suppose that  $\overline{M}$ is a 2-disk.
    If $\overline{M}$ contains at least two vertices, then the inverse image of any edge in $ \overline{M}$ is a submanifold admitting a cohomogeneity one $G$-action if and only if the angles  at its vertices are $\pi/2$.
\end{proposition}
\begin{proof}

    Let $\overline{N}$ be an edge of $\overline{M}$ with vertices $\overline{p}_1$ and $\overline{p}_2$, and let $N$ denote the corresponding submanifold with cohomogeneity one group diagram $(G,K,L_1,L_2)$. $N$ is homeomorphic to the union of disk bundles $D(G/L_1)\cup_E D(G/L_2)$, glued along their common boundary $E$. The principal orbits $G/K$ fiber over the singular orbits $G/L_1$ and $G/L_2$ with sphere fibers $L_1/K$ and $L_2/K$. Moreover, the induced action of each $L_i$ at $p_i$ on its unit normal sphere is by cohomogeneity one, with group diagram $(L_i, H, K_i,K_i')$, where $H$ is the principal isotropy subgroup of the $G$-action.

    Via the classification of cohomogeneity one actions on spheres, one sees that the only actions with singular orbits that are spheres are those for which the diameter of the orbit space is $\pi/2$. Thus $N$ is a cohomogeneity one manifold if and only if the angles at $\overline{p}_1$ and $\overline{p}_2$ are equal to $\pi/2$.
\end{proof}

Finally, we state Theorem C of Dong, Escher, and Searle \cite{DES}, which provides an equivariant diffeomorphism classification provided $M$ is rationally elliptic and the action on $M$ is isotropy-maximal. A forthcoming corrigendum corrects its statement in \cite{DES}, where the additional hypothesis that all faces are contractible was mistakenly omitted.

\begin{theorem}\cite{DES}\label{classify}
	Let $M^n$ be a closed, rationally elliptic $n$-manifold admitting a smooth, effective, locally standard, and isotropy-maximal $T^k$-action. Suppose all faces of $\overline{M}$ are contractible and all 4-dimensional faces of $\overline{M}$ are diffeomorphic to disks, after smoothing the corners. Then $M^n$ is equivariantly diffeomorphic to a quotient of a free linear torus action of $$\mathcal{Z}^m=\prod_{i<r}S^{2n_i}\times\prod_{i\geq r}S^{2n_i-1},\hspace{5pt} n_i\geq2, \hspace{5pt}where\hspace{5pt} n\leq m\leq3n-3k.$$
\end{theorem}

\begin{remark}\label{4Dfaces}
    For actions of cohomogeneity three or less, the hypothesis of 4-dimensional faces being diffeomorphic to disks is vacuous.
\end{remark}

\subsection{Alexandrov spaces}

Recall that an \textit{Alexandrov space} $X$ is a finite dimensional, locally compact, locally complete length space with a lower curvature bound in the triangle comparison sense, denoted $\curv(X)\geq k$. Alexandrov spaces are a natural generalization of Riemannian manifolds with a lower sectional curvature bound. In particular, the corollary in Section 4.6 of Burago, Gromov, and Perelman \cite{BGP} shows that orbit spaces of isometric $G$-actions on such manifolds are examples of Alexandrov spaces. For more information on Alexandrov spaces in general, see  Burago, Burago, and Ivanov \cite{BBI}.

The analog to the unit tangent sphere at a point $p\in X$ of an Alexandrov space is called the \textit{space of directions} and denoted $\Sigma_p$. By definition, $\Sigma_p$ is the completion of the space of geodesic directions at $p$ equipped with the angular metric.

When $\dim(X)\geq2$, the space of directions $\Sigma_p$ at each point $p\in X$ is a positively curved Alexandrov space with $\dim(\Sigma_p)=\dim(X)-1$. Positively curved 1-dimensional Alexandrov spaces are homeomorphic to a closed interval $I$ or $S^1$. Thus, every Alexandrov space of dimension two is a topological manifold with or without boundary, as stated in Corollary 10.10.3 of \cite{BBI}. As all $0$- and $1$-dimensional Alexandrov spaces are homeomorphic to a point, an interval, or $S^1$, we obtain the following lemma.

\begin{lemma}\cite{BBI}\label{topmfld}
    Let $X$ be an Alexandrov space of dimension at most $2$. Then $X$ is a topological manifold, possibly with boundary.
\end{lemma}

\noindent 

Consider now the orbit space $\overline{M}$ of an isometric $G$-action, where $\sec(M)\geq k$. The space of directions $\Sigma_{\overline{x}}$ is isometric to $S^\perp_x/G_x$, where $S^\perp_x$ is the normal sphere to the orbit $G(x)$ at $x$.

\subsection{Geometric results} The assumption of almost non-negative curvature provides additional structure that we require to prove \hyperref[thmA]{Theorem A}. We first recall its definition.
\begin{definition}[Almost non-negative curvature]\label{ANNCdef}
    A smooth manifold $M$ is said to be \textnormal{almost non-negatively curved} if it admits a sequence $\{g_k\}_{k=1}^\infty$ of Riemannian metrics such that $\sec(M,g_k)\geq-1/k^2$ and $\diam(M,g_k)\leq1$ for all $k$. The sequence $\{g_k\}$ is also said to be almost non-negatively curved.
\end{definition}

\noindent A primary motivation to study the class of almost non-negatively curved manifolds is that they are exactly the manifolds that collapse to a point with a lower curvature bound, as we see in the following proposition. That is, they admit a sequence of Riemannian metrics $\{h_k\}_{k=1}^\infty$ such that $\sec(M,h_k)\geq K$ for some fixed $K$ and $\lim_{k\to\infty}\diam(M,h_k)=0$.

\begin{proposition}\label{ANNC}
    A smooth manifold $M$ is almost non-negatively curved if and only if it collapses to a point while maintaining a lower sectional curvature bound. 
\end{proposition}

\begin{proof}
    Suppose first that $\{g_k\}_{k=1}^\infty$ is an almost non-negatively curved sequence of metrics on $M$. Then the sequence of metrics $\overline{g}_k$, rescaled by a factor of $1/k$, satisfies the properties $\diam(M,\overline{g}_k)=\diam(M,g_k)/k\leq1/k$ and $\sec(M,\overline{g}_k)=k\cdot\sec(M,g_k)\geq-1$. That is, $M$ collapses to a point while maintaining a lower curvature bound. Now suppose that $\{h_k\}_{k=1}^\infty$ is a sequence of Riemannian metrics on $M$ such that $\sec(M,h_k)\geq K$ and $\diam(M,h_k)\to 0$ as $k$ tends to infinity.
    Defining $d_k:=\diam(M,h_k)$, the sequence $\overline{h}_k$ rescaled by $1/d_k$ has $\diam(M,\overline{h}_k)=d_k/d_k=1$ and $\sec(M,\overline{h}_k)\geq d_k\sec(M,h_k)\geq Kd_k$. Since $d_k$ tends to $0$ as $k$ tends to infinity, there is an almost non-negatively curved subsequence of $\{\overline{h}_k\}$.
\end{proof}

Since an almost non-negatively curved manifold has a sequence of Riemannian metrics associated to it, one must adapt the notion of an isometric group action in this setting.
\begin{definition}
    An \textnormal{almost non-negatively curved $G$-manifold} is a manifold equipped with a smooth $G$-action and an almost non-negatively curved sequence $\{g_k\}_{k=1}^\infty$ of Riemannian metrics such that the $G$-action is isometric with respect to each $g_k$.
\end{definition}
\noindent Theorem 1.1 in Grove, Wilking, and Yeager \cite{GWY} provides topological information about almost non-negatively curved $G$-manifolds of low cohomogeneity.

\begin{theorem}\cite{GWY}\label{anncre}
    Let $G$ be a compact, connected Lie group and let $M$ be a closed, simply connected, almost non-negatively curved $G$-manifold. If the $G$-action on $M$ is of cohomogeneity at most two, then $M$ is rationally elliptic.
\end{theorem}

\noindent By Corollary 6.3 in Chapter 2 of Bredon \cite{Bredon}, if $G$ is connected and $M$ is simply connected, then $\overline{M}$ is simply connected. Since $\overline{M}$ is an Alexandrov space of dimension at most two, Lemma \ref{topmfld} implies that $\overline{M}$ is a topological manifold with or without boundary. Hence, $\overline{M}$ is homeomorphic to a point, an interval, $S^2$, or $D^2$. In this setting, Theorem 1.2 in \cite{GWY} applies a lifting result of Searle and Wilhelm \cite{SearleWilhelm} to relate the structure of almost non-negatively curved $G$-manifolds to the topology and geometry of $\overline{M}$.

\begin{theorem}\cite{GWY}\label{maxsides}
	Let $M$ be a closed, simply connected manifold and let $G$ be a compact, connected Lie group. Suppose $G$ acts smoothly on $M$ by cohomogeneity at most two. Then $M$ is an almost non-negatively curved $G$-manifold if and only if $\overline{M}$ is not a disk with more than 4 edges.
\end{theorem}

\section{Proof of Theorem A}
In this section, we prove Theorem A by showing that the conditions of Theorem \ref{classify} are satisfied. A key ingredient is the next proposition, which proves that the $T^{n-2}$-action is locally standard.

\begin{proposition}\label{actionlocstd}
    Let $M^n$ be a closed, simply connected $T^{n-2}$-manifold, $n\geq 4$. Then the $T^{n-2}$-action on $M$ is locally standard.
\end{proposition}
\begin{proof}
    Let $x\in M$. Recall from Definition \ref{locstd} that we must show that $x$ has a neighborhood equivariantly diffeomorphic to $T^r\times W\times\R^m$, where $r=(n-2)-\dim(T^{n-2}_x)$, $W$ is a faithful $T^{n-2}_x$-representation of real dimension $2\dim(T^{n-2}_x)$, and $T^k_x$ acts trivially on $T^r$. We proceed by cases based on the isotropy type of the orbit. By Theorem \ref{ospace}, the only possible isotropy groups are $\{e\}$, $T^1$, or $T^2$. It is clear for principal orbits that the $T^{n-2}$-action is locally standard, so we only consider the cases for singular orbits.

    Suppose that the orbit has $T^1$-isotropy. Then $T^{n-2}(x)\cong T^{n-3}$ and the normal disk is three-dimensional. By the \hyperref[ST]{Slice Theorem}, the orbit has a neighborhood equivariantly diffeomorphic to $T^{n-2}\times_{T^1} D^3$, where $T^1$ acts effectively and linearly on $D^3\simeq\R^3$. From Table 11 in \cite{McGowan}, there is only one possible induced action of $T^1\cong SO(2)$ on the normal sphere $S^2\subset\R^3$. The action on the normal disk is then the cone of the action on the normal sphere, so $T^1$ acts by rotations on the $\R^2$ factor of $\R^3\simeq\R^2\times\R^1$ and trivially on the $\R^1$ factor. Hence, the torus action is locally standard.

    Now suppose that the orbit has $T^2$-isotropy. Then $T^{n-2}(x)\cong T^{n-4}$ and the normal disk is four-dimensional. By the \hyperref[ST]{Slice Theorem}, the orbit has a neighborhood equivarianly diffeomorphic to $T^{n-2}\times_{T^2} D^4$, with the isotropy group $T^2$ acting effectively and linearly on $D^4\simeq\R^4$. We again identify the induced action of $T^2\cong SO(2)\times SO(2)$ on the normal sphere $S^3\subset\R^4$ from Table 11 in \cite{McGowan}. By coning this action, we see that $T^2\cong T^1\times T^1$ acts on $\R^4\cong\R^2\times\R^2$, where each circle factor of $T^2$ acts by rotations on the corresponding $\R^2$ factor and trivially on the other $\R^2$ factor. Therefore the action is locally standard.
\end{proof}
We are now in a position to simultaneously prove Theorem A and Corollary B.

\begin{proof}[Proof of Theorem A and Corollary B]
    We begin by proving Part (1). By Lemma \ref{vertices}, there are orbits with $T^2$-isotropy, so the action is isotropy-maximal. Lemma \ref{maxtorus} then implies that $n-2$ is the maximal symmetry rank, giving us the result.
    
    We now prove Part (2). It suffices to show that the conditions of Theorem \ref{classify} hold. That is, we must show that $M$ is rationally elliptic, the $T^{n-2}$-action is locally standard and isotropy-maximal, all faces of $\overline{M}$ are contractible, and all $4$-dimensional faces of $\overline{M}$ are diffeomorphic to disks. By Remark \ref{4Dfaces}, the last condition holds vacuously, so we proceed to show that the rest of these conditions hold simultaneously for dimensions four, five, and six. 
    
    By Theorem \ref{ospace}, $\overline{M}$ is homeomorphic to $D^2$. Applying Lemma \ref{vertices} and Theorem \ref{maxsides} yields that $\overline{M}$ has at least $n-2$ edges and at most 4. When $n\geq7$, this yields a contradiction, proving \hyperref[corB]{Corollary B}. For the remainder of the proof, we assume $n\leq6$. The possible orbit spaces are shown in Figure \ref{pictures}.

\begin{figure}[H] 
	\hspace*{\fill}
	\begin{subfigure}[b]{0.3\textwidth}
			\centering
			\includegraphics[width=0.7\textwidth]{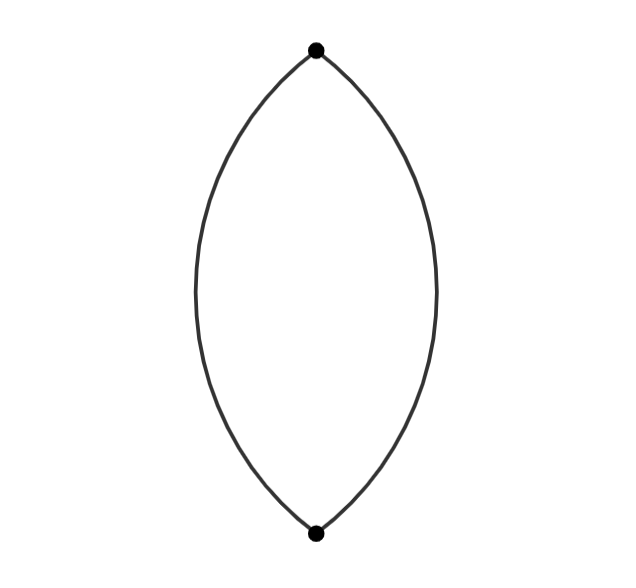}
			\caption{}
		\end{subfigure}
	\hfill
	\begin{subfigure}[b]{0.3\textwidth}
			\centering
			\includegraphics[width=0.75\textwidth]{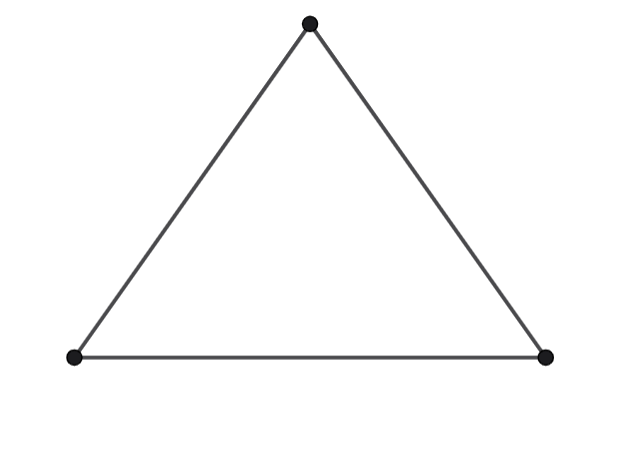}
			\caption{}
		\end{subfigure}
	\hfill
	\begin{subfigure}{0.3\textwidth}
			\centering
			\includegraphics[width=0.75\textwidth]{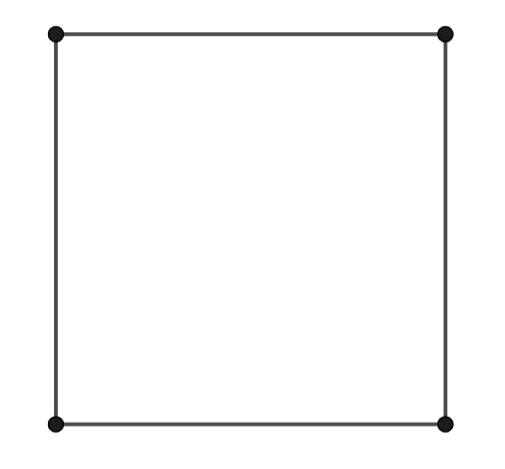}
			\caption{}
		\end{subfigure}
	\hspace*{\fill}
	\caption{Possible orbit spaces for Theorem A.}\label{pictures}
\end{figure}

\noindent Note that the angle at each vertex measures $\pi/2$ since the space of directions at a vertex is an interval isometric to $S^2/S^1$. Table 11 in \cite{McGowan} then implies that this quotient has length $\pi/2$. By Lemma \ref{vertices}, the lune (A) appears only when $n=4$, the triangle (B) when $n=4$ or 5, and the quadrilateral (C) when $n=4,5,$ or 6.

Now suppose that $M$ is an almost non-negatively curved $T^{n-2}$-manifold. By Proposition \ref{actionlocstd}, the $T^{n-2}$-action is locally standard. Since the action is by cohomogeneity two, Theorem \ref{anncre} implies that $M$ is rationally elliptic. Finally, all faces of $\overline{M}$ are contractible. Therefore Theorem \ref{classify} provides the desired classification and the proof is complete.
\end{proof}

\bibliographystyle{plain}

\bibliography{ANNCreferences}

@mastersthesis{Bartel,
    author = {Samuel Bartel},
    title = {\textit{Almost non-negative curvature and torus symmetry in low dimensions}},
    school = {Wichita State University},
    year = {2024}
}

@article{BGP,
    author = {Yuri Burago and Mikhael Gromov and Grigori Perelman},
    title = {A.D. Alexandrov spaces with curvature bounded below},
    journal = {Russian Math. Surveys},
    year = {1992},
    volume = {47},
    number = {2}
}

@book{Bredon,
    author = {Glen Bredon},
    title = {Introduction to compact transformation groups},
    series = {Pure and Applied Mathematics},
    number = {42},
    publisher = {Academic Press, New York NY},
    year = {1972}
}

@book{BBI,
    author = {Dmitri Burago and Yuri Burago and Sergei Ivanov},
    title = {A Course in Metric Geometry},
    series = {Graduate Studies in Mathematics},
    volume = {33},
    publisher = {Amer. Math. Soc.},
    year = 2001
}

@article{DES,
	title={Almost isotropy-maximal manifolds of non-negative curvature}, 
	author={Zheting Dong and Christine Escher and Catherine Searle},
	year={2024},
	journal={Trans. Amer. Math. Soc.},
    volume={\textbf{377}},
    number={7},
    doi={https://doi.org/10.1090/tran/9100}
}

@article{Escher_Searle,
    author = {Christine Escher and Catherine Searle},
    title = {Torus actions, maximality, and non-negative curvature},
    journal = {J. reine angew. Math.},
    year = {2021},
    volume = {\textbf{780}}
}

@article{Galaz_Garcia_2012,
    author = {Fernando {Galaz-Garc\'ia}},
    title = {Nonnegatively curved fixed point homogeneous manifolds in low dimensions},
    journal = {Geomitriae Dedicata},
    year = {2012},
    volume = {\textbf{157}}
}

@article{Galaz_Garcia_Searle_2011,
	URL = {http://www.jstor.org/stable/41291817},
	author = {Fernando Galaz-Garc\'ia and Catherine Searle},
	journal = {Proc. Amer. Math. Soc.},
	number = {7},
	publisher = {American Mathematical Society},
	title = {Low-dimensional manifolds with non-negative curvature and maximal symmetry rank},
	volume = {\textbf{139}},
	year = {2011}
}

@article{Galaz_Garcia_Kerin,
    title = {Cohomogeneity-two torus actions on non-negatively curved manifolds of low dimension},
    author = {Galaz-Garc\'ia, Fernando and Kerin, Martin},
    year = {2013},
    volume = {\textbf{276}},
    journal = {Mathematische Zeitschrift},
    doi = {10.1007/s00209-013-1190-5}
}

@article{Galaz_Garcia_Spindeler,
    title = {Nonnegatively curved fixed point homogeneous 5-manifolds},
    author = {Fernando {Galaz-Garc\'ia} and Wolfgang Spindeler},
    year = {2012},
    journal = {Annals of Global Analysis and Geometry},
    volume = {\textbf{41}}
}

@inproceedings{Grove_2002,
	author = {Grove, Karsten},
	year = {2002},
	pages = {31-53},
	title = {Geometry of, and via, symmetries},
    booktitle = {Conformal, Riemannian and Lagrangian geometry (Knoxville, TN, 2000)},
    series = {Univ. Lecture Ser.},
	volume = {27},
    publisher = {Amer. Math. Soc.},
	isbn = {9780821832103},
	doi = {10.1090/ulect/027/02}
}

@article{Grove_Searle_97,
    author = {Karsten Grove and Catherine Searle},
    title = {Differential topological restrictions by curvature and symmetry},
    journal = {J. Diff. Geo.},
    year = {1997},
    volume = {\textbf{47}}
}

@article{Harvey_Searle,
    author = {John Harvey and Catherine Searle},
    title = {Almost non-negatively curved 4-manifolds with torus symmetry},
    journal = {Proc. Amer. Math. Soc.},
    year = 2020,
    volume = {\textbf{148}},
    number = 11
}

@article{Hermann,
    author = {Martin Hermann},
    title = {Classification and characterization of rationally elliptic manifolds in low dimensions},
    journal = {Mathematische Zeitschrift},
    volume = {\textbf{288}},
    year = {2018},
    doi = {https://doi.org/10.1007/s00209-017-1927-7}
}

@article{Ishida_2019,
	url = {https://doi.org/10.1515/crelle-2016-0023},
	title = {Complex manifolds with maximal torus actions},
	author = {Hiroaki Ishida},
	volume = {\textbf{751}},
	journal = {J. reine angew. Math.},
	doi = {doi:10.1515/crelle-2016-0023},
	year = {2019},
	lastchecked = {2024-04-15}
}

@article {Kim,
	AUTHOR = {Kim, Soon Kyu and McGavran, Dennis and Pak, Jingyal},
	TITLE = {Torus group actions on simply connected manifolds},
	JOURNAL = {Pacific J. Math.},
	FJOURNAL = {Pacific Journal of Mathematics},
	VOLUME = {\textbf{53}},
	YEAR = {1974},
	ISSN = {0030-8730,1945-5844},
	MRCLASS = {57E15},
	MRNUMBER = {368051},
	MRREVIEWER = {L.\ Lininger},
	URL = {http://projecteuclid.org/euclid.pjm/1102911611},
}

@book{Kobayashi,
	author = {Shoshichi Kobayashi},
	title = {Transformation Groups in Differential Geometry},
	series = {Classics in Mathematics},
	publisher = {Springer Berlin, Heidelburg},
	year = {1995}
}

@article{McGowan,
	title = {How tightly can you fold a sphere?},
	journal = {Diff. Geom. Appl.},
	volume = {\textbf{22}},
	number = {1},
	year = {2005},
	issn = {0926-2245},
	doi = {https://doi.org/10.1016/j.difgeo.2004.07.007},
	url = {https://www.sciencedirect.com/science/article/pii/S0926224504000543},
	author = {Jill McGowan and Catherine Searle}
}

@article{Mostert,
	ISSN = {0003486X},
	URL = {http://www.jstor.org/stable/1970056},
	author = {Paul S. Mostert},
	journal = {Annals of Mathematics},
	number = {3},
	publisher = {Annals of Mathematics},
	title = {On a Compact Lie Group Acting on a Manifold},
	urldate = {2024-04-01},
	volume = {\textbf{65}},
	year = {1957}
}

@article{Paternain_Petean,
    author = {Paternain, Gabriel and Petean, Jimmy},
    title = {Zero entropy and bounded topology},
    journal = {Comment. Math. Helv.},
    volume = {\textbf{81}},
    number = {2},
    year = {2006}
}

@article{ShioyaYamaguchi,
  title={Collapsing Three-Manifolds Under a Lower Curvature Bound},
  author={Takashi Shioya and Takao Yamaguchi},
  journal={J. Diff. Geo.},
  year={2000},
  volume={\textbf{56}},
  number={1},
  url={https://api.semanticscholar.org/CorpusID:122671636}
}

@article{GWY,
    author =  {Karsten Grove and Burkhard Wilking and Joseph Yeager},
    title = {Almost non-negative curvature and rational ellipticity in cohomogeneity two},
    volume={\textbf{69}},
    number={7},
    journal ={Annales de L'Institut Fourier} ,
    year = {2019}
}

@article{SearleWilhelm,
    author = {Catherine Searle and Frederick Wilhelm},
    title = {How to lift positive Ricci curvature},
    journal = {Geometry and Topology},
    volume = {\textbf{19}},
    number ={3},
    year = {2015}
}

@article{Wiemeler,
    author = {Michael Wiemeler},
    title = {Torus manifolds and non-negative curvature},
    journal = {J. London Math. Soc.},
    volume = {\textbf{91}},
    number = {3},
    year = {2015}
}

\end{document}